\documentclass[12pt,a4paper]{article}
\usepackage{amssymb}
\usepackage{amsmath}
\usepackage{amsthm}
\usepackage{ot1patch}
\usepackage{indentfirst}

\newtheorem{tw}{Theorem}[section]
\newtheorem{pr}[tw]{Proposition}
\newtheorem{lm}[tw]{Lemma}
\newtheorem{cor}[tw]{Corollary}

\DeclareMathOperator{\cha}{char}
\DeclareMathOperator{\jac}{jac}

\DeclareMathOperator{\dgcd}{dgcd}

\author{Piotr J\k edrzejewicz\\
\normalsize Faculty of Mathematics and Computer Science\\
\normalsize Nicolaus Copernicus University\\
\normalsize Toru\'{n}, Poland}
\title{Irreducible Jacobian derivations\\
in positive characteristic}
\date{}

\begin{document}

\maketitle

\begin{abstract}
We prove that an irreducible polynomial derivation
in positive characteristic
is a Jacobian derivation if and only if
there exists an $n-1$-element $p$-basis
of its ring of constants.
In the case of two variables we characterize
these derivations in terms of their divergence
and some nontrivial constants.
\end{abstract}

\begin{table}[b]\footnotesize\hrule\vspace{1mm}
\end{table}

\section{Introduction}

Nowicki and Nagata proved in \cite{NN}, Theorem~2.8
(see also \cite{polder}, Theorem~7.1.4),
that the ring of constants of a polynomial derivation
in two variables over a field of zero characteristic
is generated by a single element.
They also obtained an analog of this fact
in characteristic $2$ (\cite{NN}, Proposition~4.2).
Namely, they showed that the ring of constants
of a $k$-derivation of the polynomial algebra $k[x,y]$,
where $k$ is a field of characteristic $2$,
is generated by a single element over $k[x^2,y^2]$.
Nowicki and Nagata observed that this property
does not hold in characteristic $p>2$
(\cite{NN}, Example~4.3).

\medskip

In \cite{phomogen}, Theorem~5.6, the present author obtained
a characterization of $p$-ho\-mo\-gene\-ous derivations
of the polynomial algebra $k[x,y]$,
where $k$ is a field of characteristic $p>0$,
such that the ring of constants is generated
by a single element over $k[x^p,y^p]$.
These results were partially generalized to derivations
homogeneous with respect to arbitrary weights
in \cite{ontwovar}, Theorem 11.
The characterization of such single generators
(that is, one-element $p$-bases)
was obtained in \cite{charoneel}, Theorem~4.2,
while the characterization of many-element $p$-bases
was obtained in \cite{charpbases}, Theorem~4.4.
In this paper we apply these new results
from \cite{charoneel} and \cite{charpbases}
to some older questions
from \cite{phomogen} and~\cite{ontwovar}.

\medskip

In Theorem~\ref{t1} we consider a $K$-derivation $d$
of the polynomial algebra $K[x_1,\dots,x_n]$
over a unique factorization domain $K$ of characteristic $p>0$,
such that $d(x_1)$, $\dots$, $d(x_n)$ are coprime.
We prove that $d$ is a Jacobian derivation
if and only if its ring of constants has
a $p$-basis over $K[x_1^p,\dots,x_n^p]$
consisting of $n-1$ elements.
In Theorem~\ref{t2} we obtain a generalization
of Theorem~5.6 from \cite{phomogen}
to arbitrary polynomial derivations in two variables
in positive characteristic.

\section{Preliminary definitions and facts}

Throughout this paper by a domain we mean
a commutative ring with unity, without zero divisors.
Let $A$ be a domain of characteristic $p>0$.

\medskip

If $d$ is a derivation of a $A$,
then its kernel is called the ring of constants,
and is denoted by $A^d$.
Note that if $A$ is a $K$-algebra,
where $K$ is a domain of characteristic $p>0$,
then every $K$-derivation of $A$ is a $KA^p$-derivation,
where $A^p=\{a^p,\,a\in A\}$.
Recall (\cite{phomogen}, Theorem~1.1
and \cite{eigen}, Theorem~2.5)
that if $A$ is finitely generated as a $K$-algebra,
then a subring $R\subset A$ is ring of constants
of some $K$-derivation of $A$
if and only if it satisfies the conditions
$$KA^p\subset R\hspace{3mm}\mbox{and}\hspace{3mm}
R_0\cap A=R,$$
where $R_0$ denotes the field of fractions of $R$.
Note also that every $K$-derivation $d$ of $A$
can be uniquely extended to a $K_0$-derivation $\delta$
of the field $A_0$ in such a way that
$\delta\big(\frac{a}{b}\big)=\frac{ad(b)-bd(a)}{b^2}$
for $a,b\in A$, $b\neq 0$.
If $d$ is a $K$-derivation of the polynomial $K$-algebra
$A=K[x_1,\dots,x_n]$, where $K$ is a UFD,
and $g=\gcd(d(x_1),\dots,d(x_n))$,
then $d'=\frac1{g}d$ is an irreducible $K$-derivation of $A$,
that is, $d'(x_1)$, $\dots$, $d'(x_n)$ are coprime.
Note that derivations $d$ and $d'$ have
the same ring of constants,
so, if we are interested in rings of constants,
it is enough to consider irreducible derivations.

\medskip

Let $B$ be a subring of $A$, containing $A^p$.
Recall the definitions of $p$-independence
and a $p$-basis (\cite{Matsumura}, p.~269).
The elements $z_1,\dots,z_m\in A$ are called
$p$-independent over $B$, if the elements
of the form $z_1^{i_1}\ldots z_m^{i_m}$,
where $0\leqslant i_1,\dots,i_m<p$,
are linearly independent over $B$.
Note that $z_1,\dots,z_m$ are $p$-independent over $B$
if and only if the degree of the field extension
$B_0\subset B_0(z_1,\dots,z_m)$ equals $p^m$.
The elements $z_1,\dots,z_m\in A$ are called
a $p$-basis of $A$ over $B$, if the elements
of the form $z_1^{i_1}\ldots z_m^{i_m}$,
where $0\leqslant i_1,\dots,i_m<p$,
form a basis of $A$ as a $B$-module.
Note that $z_1,\dots,z_m$
form a $p$-basis of $A$ over $B$
if and only if they are $p$-independent over $B$
and generate $A$ as a $B$-algebra.
If $A$ and $B$ are fields,
the degree of the extension $B\subset A$ equals $p^m$,
and elements $z_1,\dots,z_m\in A$
are $p$-independent over $B$,
then $z_1,\dots,z_m$ form a $p$-basis of $A$ over $B$.
If $z_1,\dots,z_m$ form a $p$-basis of the domain $A$
over $B$, then every element $a\in A$
can be uniquely presented in the form
$$a=\sum_{0\leqslant i_1,\dots,i_m<p}b_{i_1,\dots,i_m}
z_1^{i_1}\ldots z_m^{i_m},$$
where $b_{i_1,\dots,i_m}\in B$.
Moreover, if $z_1,\dots,z_m$ form a $p$-basis of $A$ over $B$,
then for every $g_1,\dots,g_m\in A$ there exists
a unique $B$-derivation $d$ of $A$ such that
$d(z_i)=g_i$ for $i=1,\dots,m$.
So, if $d$ and $d'$ are $B$-derivations of $A$
such that $d(z_i)=d'(z_i)$ for $i=1,\dots,m$,
then $d=d'$.

\medskip

Let $A=K[x_1,\dots,x_n]$ be the polynomial $K$-algebra
in $n$ variables, where $K$ is a UFD of characteristic $p>0$,
put $B=K[x_1^p,\dots,x_n^p]$.
Consider polynomials $f_1,\dots,f_m\in A$,
where $m\geqslant 1$.
For arbitrary $j_1,\dots,j_m\in\{1,\dots,n\}$
denote by $\jac^{f_1,\dots,f_m}_{j_1,\dots,j_m}$
the Jacobian determinant of $f_1,\dots,f_m$
with respect to $x_{j_1},\dots,x_{j_m}$.
Following \cite{jaccond},
we define a differential gcd of $f_1,\dots,f_m$:
$$\dgcd(f_1,\dots,f_m)=
\gcd\left(\jac^{f_1,\dots,f_m}_{j_1,\dots,j_m},\;
j_1,\dots,j_m\in\{1,\dots,n\}\right).$$
We put $\dgcd(f_1,\dots,f_m)=0$
if $\jac^{f_1,\dots,f_m}_{j_1,\dots,j_m}=0$
for every $j_1,$ $\dots,$ $j_m\in\{1,$ $\dots,n\}$,
that is, $f_1,\dots,f_m$ are $p$-dependent over $B$
(this equivalence follows from \cite{charpbases},
Lemma~3.3, Proposition~3.5 and Example~3.1).
Observe that in the case of $m=n-1$ we have
$$\dgcd(f_1,\dots,f_{n-1})=
\gcd\big(\jac^{f_1,\dots,f_{n-1}}_{1,\dots,\widehat{i},\dots,n},
\;i=1,\dots,n\big),$$
where $\widehat{i}$ means that the element $i$ is omitted.
Note that $\dgcd(f_1,\dots,f_m)$ is determined up to
associativity in $A$.
Two polynomials $g,h\in A$ are called associated,
and we denote it $g\sim h$,
if $g=ah$ for some invertible element $a\in K$.

\medskip

Let $d$ be a $K$-derivation of the polynomial algebra
$K[x_1,\dots,x_n]$ over a domain $K$
or an $L$-derivation of the field of rational functions
$L(x_1,\dots,x_n)$ over a field $L$.
The polynomial
$$d^{\ast}=\frac{\partial (d(x_1))}{\partial x_1}+\ldots+
\frac{\partial (d(x_n))}{\partial x_n}$$
is called the divergence of $d$
(see \cite{polder},~2.3).
For any element $f$ belonging to $K[x_1,\dots,x_n]$,
resp. $L(x_1,\dots,x_n)$,
we have $(fd)^{\ast}=fd^{\ast}+d(f)$.
Hence, if $d(f)=0$, then $(fd)^{\ast}=fd^{\ast}$.

\medskip

Given polynomials $f_1,\dots,f_{n-1}\in K[x_1,\dots,x_n]$,
where $K$ is a domain,
put $F=(f_1,\dots,f_{n-1})$ and consider
a $K$-derivation $d_F$ of $K[x_1,\dots,x_n]$ such that
$$d_F(g)=\jac(f_1,\dots,f_{n-1},g)$$
for $g\in K[x_1,\dots,x_n]$,
where $\jac$ denotes the usual Jacobian determinant
with respect to $x_1,\dots,x_n$.
A derivation of the form $d_F$ is called
a Jacobian derivation.
In a similar way we define a Jacobian derivation
of the field of rational functions $L(x_1,\dots,x_n)$
over a field $L$.
Finally, recall that a Jacobian derivation has zero divergence
(\cite{Freudenburg}, p.~58, Lemma~3.8,
the arguments are characteristic-free).

\section{\boldmath Jacobian derivations in $n$ variables}

The following lemma is a positive characteristic analog
of Lemma~6 from~\cite{MakarLimanov}
(see also \cite{Freudenburg}, p.~57, Lemma~3.6).

\begin{lm}
\label{l1}
Let $L$ be a field of characteristic $p>0$.
Assume that rational functions
$f_1,\dots,f_{n-1}\in L(x_1,\dots,x_n)$
are $p$-in\-de\-pen\-dent over $L(x_1^p,\dots,x_n^p)$.
Let $d$ be an $L$-derivation of $L(x_1,\dots,x_n)$
such that $d(f_1)=\dots=d(f_{n-1})=0$.
Then $d=cd_F$ for some $c\in L(x_1,\dots,x_n)$,
where $F=(f_1,\dots,f_{n-1})$.
\end{lm}

\begin{proof}
Put $M=L(x_1^p,\dots,x_n^p)$.
Choose an element
$g\in L(x_1,\dots,x_n)\setminus M(f_1,\dots,f_{n-1})$
and put $c=\frac{d(g)}{d_F(g)}$,
where $d_F(g)=\jac(f_1,\dots,f_{n-1},g)\neq 0$,
because $f_1,\dots,f_{n-1},g$ are $p$-independent over $M$.
The elements $f_1,\dots,f_{n-1},g$ form a $p$-basis
of $L(x_1,\dots,x_n)$ over $M$,
and we have $d(g)=cd_F(g)$
and $d(f_i)=0=cd_F(f_i)$ for $i=1,\dots,n-1$,
so $d=cd_F$.
\end{proof}

We obtain the following generalization
of Proposition~4.1 from~\cite{phomogen}.

\begin{cor}
\label{c1}
Let $K$ be a UFD of characteristic $p>0$.
Assume that polynomials
$f_1,\dots,f_{n-1}\in K[x_1,\dots,x_n]$
are $p$-in\-de\-pen\-dent over $K[x_1^p,\dots,x_n^p]$,
put $\overline{F}=\dgcd(f_1,\dots,f_{n-1})$.
Let $d$ be a $K$-derivation of $K[x_1,\dots,x_n]$
such that $d(f_1)=\dots=d(f_{n-1})=0$.
Then $$\overline{F}d=hd_F$$
for some polynomial $h\in K[x_1,\dots,x_n]$
such that $h\sim\gcd(d(x_1),\dots,d(x_n))$.
\end{cor}

\begin{proof}
Put $A=K[x_1,\dots,x_n]$ and $B=K[x_1^p,\dots,x_n^p]$.
Consider extensions of $d$ and $d_F$
to $B_0$-derivations of $A_0$:
$\delta$ and $\delta_F$, respectively.
By Lemma~\ref{l1}, $\delta=c\delta_F$
for some $c\in A_0$.
Put $c=\frac{a}{b}$, where $a,b\in A$,
$b\neq 0$, $\gcd(a,b)\sim 1$.
Then $bd=ad_F$.
Now, observe that $b\mid d_F(x_i)$ for $i=1,\dots,n$,
so $b\mid \overline{F}$, that is,
$\overline{F}=bh'$ for some $h'\in A$.
We obtain $\overline{F}d=bh'd=ah'd_F=hd_F$,
where $h=ah'$, and we have then
$$\overline{F}\cdot\gcd(d(x_1),\dots,d(x_n))\sim
h\cdot\gcd(d_F(x_1),\dots,d_F(x_n)).$$
Recall that $\overline{F}\neq 0$, because
$f_1,\dots,f_{n-1}$ are $p$-independent over $B$.
Finally, $h\sim\gcd(d(x_1),\dots,d(x_n))$.
\end{proof}

The following property of Jacobian derivations
has been observed by Makar-Limanov
in \cite{MakarLimanov}, Lemma~7
(see \cite{Freudenburg}, p.~57, Lemma~3.7)
in the case of characteristic~$0$.

\begin{cor}
\label{c2}
Let $L$ be a field of characteristic $p>0$.
Assume that rational functions
$f_1,\dots,f_{n-1}\in L(x_1,\dots,x_n)$
are $p$-in\-de\-pen\-dent over $L(x_1^p,\dots,x_n^p)$.
Consider rational functions
$g_1,\dots,g_{n-1}\in L(x_1,\dots,x_n)^{d_F}$,
where $F=(f_1,$ $\dots,$ $f_{n-1})$.
Then $d_G=cd_F$ for some $c\in L(x_1,\dots,x_n)^{d_F}$,
where $G=(g_1,\dots,g_{n-1})$.
\end{cor}

\begin{proof}
If $g_1,\dots,g_{n-1}$ are $p$-de\-pen\-dent
over $L(x_1^p,\dots,x_n^p)$, we put $c=0$.
Now, assume that $g_1,\dots,g_{n-1}$
are $p$-in\-de\-pen\-dent over $L(x_1^p,\dots,x_n^p)$.
By Lemma~\ref{l1}, $\delta_F=a\delta_G$
for some $a\in L(x_1,\dots,x_n)$.
Note that $a\neq 0$, because $f_1,\dots,f_{n-1}$
are $p$-in\-de\-pen\-dent over $L(x_1^p,\dots,x_n^p)$.
Hence, $d_G=cd_F$ for $c=a^{-1}$.
Now, comparing the divergence, we obtain
$$0=(d_G)^{\ast}=(cd_F)^{\ast}=c(d_F)^{\ast}+d_F(c)=d_F(c),$$
so $c\in L(x_1,\dots,x_n)^{d_F}$.
\end{proof}

In the following theorem we obtain a characterization
of derivations with $n-1$-element $p$-basis
of the ring of constants.

\begin{tw}
\label{t1}
Let $K$ be a UFD of characteristic $p>0$.
Let $d$ be a nonzero $K$-derivation of
the polynomial $K$-algebra $K[x_1,\dots,x_n]$
such that $d(x_1)$, $\dots$, $d(x_n)$
are (in common) coprime.
The following conditions are equivalent:

\smallskip

\noindent
$(1)$ \
$K[x_1,\dots,x_n]^d$ has a $p$-basis
over $K[x_1^p,\dots,x_n^p]$
consisting of $n-1$ elements,

\medskip

\noindent
$(2)$ \
$d$ is a Jacobian derivation,

\medskip

\noindent
$(3)$ \
there exist polynomials
$f_1,\dots,f_{n-1}\in K[x_1,\dots,x_n]^d$
such that $$\dgcd(f_1,\dots,f_{n-1})\sim 1.$$
\end{tw}

\begin{proof}
$(1)\Rightarrow (3)$ \
was established in \cite{charpbases}, Theorem~4.4.

\medskip

\noindent
$(3)\Rightarrow (1)$ \
Put $A=K[x_1,\dots,x_n]$, $B=K[x_1^p,\dots,x_n^p]$.
Assume that $\dgcd(f_1,\dots,f_{n-1})\sim 1$
for some $f_1,\dots,f_{n-1}\in A^d$.
Then, by \cite{charpbases}, Theorem~4.4,
$f_1,\dots,f_{n-1}$ form a $p$-basis
of a ring of constants $R$ of some $K$-derivation of $A$.
The degree of the field extension $B_0\subset R_0$
equals $p^{n-1}$, because $f_1,\dots,f_{n-1}$
are $p$-independent over $B$.
Then, since $R\subset A^d$, the degree of
the field extension $(A^d)_0\subset A_0$
does not exceed $p$, so it equals $p$,
because $d$ is nonzero.
We obtain that $R_0=(A^d)_0$,
so $R=R_0\cap A=(A^d)_0\cap A=A^d$.

\medskip

\noindent
$(2)\Rightarrow (3)$ \
If $d=d_F$, where $F=(f_1,\dots,f_{n-1})$,
then $$d(x_i)=(-1)^{n+i}
\jac^{f_1,\dots,f_{n-1}}_{1,\dots,\widehat{i},\dots,n}$$
for $i=1,\dots,n$.
Hence, $\dgcd(f_1,\dots,f_{n-1})\sim 1$,
because $d(x_1)$, $\dots$, $d(x_n)$ are (in common) coprime.

\medskip

\noindent
$(3)\Rightarrow (2)$ \
Assume that $\dgcd(f_1,\dots,f_{n-1})\sim 1$
for some $f_1,\dots,f_{n-1}\in K[x_1,\dots,x_n]^d$.
By Corollary~\ref{c1}, $\overline{F}d=hd_F$
for some $h\in K[x_1,\dots,x_n]$
such that $h\sim\gcd(d(x_1),\dots,d(x_n))$,
where $\overline{F}=\dgcd(f_1,\dots,f_{n-1})$.
We have $\overline{F}\sim 1$ and $h\sim 1$,
so $d=d_{F'}$ for $F'=(h\overline{F}^{-1}f_1,
\dots,h\overline{F}^{-1}f_{n-1})$.
\end{proof}

\section{Jacobian derivations in two variables}

In this section we consider the polynomial algebra $K[x,y]$,
where $K$ is a UFD of characteristic $p>0$.

\medskip

Recall Lemma~5.1 and Proposition~5.4 from~\cite{phomogen}.
Note that $K$ was a field in the original formulations,
but the proofs are valid for $K$ being a UFD.

\begin{lm}
\label{l2}
Let $d$ be a $K$-derivation of $K[x,y]$ and let
$$d(x)=\sum_{0\leqslant i,j<p}a_{ij}x^iy^j,\hspace{3mm}
d(y)=\sum_{0\leqslant i,j<p}b_{ij}x^iy^j,$$
where $a_{ij},b_{ij}\in K[x^p,y^p]$.

\medskip

\noindent
{\bf a)} \
Then $d$ is a Jacobian derivation if and only if
$d^{\ast}=0$, $a_{0,p-1}=0$ and $b_{p-1,0}=0$.

\medskip

\noindent
{\bf b)} \
Let $d(x)$, $d(y)$ be coprime and $d^{\ast}=0$.
Then $K[x,y]^d=K[x^p,y^p]$ if and only if
$a_{0,p-1}\neq 0$ or $b_{p-1,0}\neq 0$.
\end{lm}

\begin{cor}
\label{c3}
If $d$ is a $K$-derivation of $K[x,y]$ such that
$d(x)$, $d(y)$ are coprime, $d^{\ast}=0$
and $K[x,y]^d\neq K[x^p,y^p]$,
then $d$ is a Jacobian derivation.
\end{cor}

For every polynomial $f\in K[x,y]$
consider a presentation in the form
$$f=\sum_{0\leqslant i,j<p}a_{ij}x^iy^j,$$
where $a_{ij}\in K[x^p,y^p]$ for $i,j\in\{0,\dots,p-1\}$.
Denote: $$f_{(p)}=a_{00}.$$
Note that $d(f_{(p)})=0$
for every $K$-derivation $d$ of $K[x,y]$.
Moreover, $f\in K[x^p,y^p]$ if and only if $f_{(p)}=f$.
Observe also that $(f+g)_{(p)}=f_{(p)}+g_{(p)}$
for two polynomials $f,g\in K[x,y]$
and that $(h\cdot f)_{(p)}=h\cdot f_{(p)}$
for $h\in K[x^p,y^p]$ and $f\in K[x,y]$.

\medskip

The following lemma will be useful for constructing
a one-element $p$-basis in Proposition~\ref{p1}.
Some motivations of this approach come from
\cite{Ono}, Theorem~4.1.

\begin{lm}
\label{l3}
Let $d$ be a nonzero $K$-derivation
of the polynomial $K$-algebra $K[x,y]$,
where $K$ is a UFD, $\cha K=p>0$.
Assume that $d^{\ast}=0$, $K[x,y]^d\neq K[x^p,y^p]$,
and that $d(x)$, $d(y)$ are coprime.
Consider a polynomial $g\in K[x,y]$ such that
$g-g_{(p)}$ is a minimal nonzero polynomial
(with respect to an ordinary degree) belonging to $K[x,y]^d$.
Put $\overline{g}=
\gcd\big(\frac{\partial g}{\partial x},
\frac{\partial g}{\partial y}\big)$.
Then:

\medskip

\noindent
{\bf a)} \
$\overline{g}d=hd_g$,
where $h$ is an invertible element of $K$,

\medskip

\noindent
{\bf b)} \
$\overline{g}\in K[x^p,y^p]$.
\end{lm}

\begin{proof}
Observe that $d(g)=d(g-g_{(p)})=0$
and that $g$ is $p$-independent over $K[x^p,y^p]$,
so from Corollary~\ref{c1} we obtain that
$\overline{g}d=hd_g$ for some $h\in K[x,y]$
such that $h\sim\gcd(d(x),d(y))\sim 1$,
so $h$ is an invertible element of $K$.
Hence
$$d(\overline{g})=d(\overline{g})+\overline{g}d^{\ast}=
(\overline{g}d)^{\ast}=hd_g^{\ast}=0,$$
so $\overline{g}-\overline{g}_{(p)}\in K[x,y]^d$.
By the assumption, $g-g_{(p)}\neq 0$,
so $g\not\in K[x^p,y^p]$,
and then $\frac{\partial g}{\partial x}\neq 0$
or $\frac{\partial g}{\partial y}\neq 0$.
If, for example, $\frac{\partial g}{\partial x}\neq 0$,
then
$$\deg (\overline{g}-\overline{g}_{(p)})\leqslant
\deg \overline{g}\leqslant
\deg\frac{\partial g}{\partial x}=
\deg\frac{\partial}{\partial x}(g-g_{(p)})<
\deg (g-g_{(p)}).$$
By the minimality of $g-g_{(p)}$
we infer that $\overline{g}-\overline{g}_{(p)}=0$,
so $\overline{g}\in K[x^p,y^p]$.
\end{proof}

In the next proposition and Theorem~\ref{t2}
the results of \cite{phomogen} (Theorem~5.6)
and \cite{ontwovar} (Theorem~11, Corollary~12)
are generalized to arbitrary polynomial derivations
in two variables in positive characteristic.

\begin{pr}
\label{p1}
Let $K$ be a UFD of characteristic $p>0$.
Let $d$ be a nonzero $K$-derivation
of the polynomial $K$-algebra $K[x,y]$
such that $d(x)$ and $d(y)$ are coprime.
Consider a polynomial $f\in K[x,y]\setminus K[x^p,y^p]$.
The following conditions are equivalent:

\smallskip

\noindent
$(1)$ \
$K[x,y]^d=K[x^p,y^p,f]$,

\medskip

\noindent
$(2)$ \
$d=cd_f$ for some invertible element $c\in K$,

\medskip

\noindent
$(3)$ \
$d(f)=0$ and $\gcd\big(\frac{\partial f}{\partial x},
\frac{\partial f}{\partial y}\big)\sim 1$,

\medskip

\noindent
$(4)$ \
$d^{\ast}=0$, $f-f_{(p)}$ is a minimal nonzero polynomial
(with respect to an ordinary degree) belonging to $K[x,y]^d$
and $f-f_{(p)}$ is not divisible by any non-invertible
element of $K\setminus\{0\}$.
\end{pr}

\begin{proof}
The equivalence of conditions $(1)$, $(2)$ and $(3)$
follows from the proof of Theorem~\ref{t1}
in the case of $n=2$.

\medskip

\noindent
$(2)\Rightarrow (4)$ \
Assume that $d=cd_f$ for some invertible element $c\in K$.
Note that $d^{\ast}=0$ and $K[x,y]^d\neq K[x^p,y^p]$.
Let $g$ be as in Lemma~\ref{l3}, so $\overline{g}d=hd_g$,
where $h$ is an invertible element of $K$.
Hence
$$d_{c\overline{g}f}=c\overline{g}d_f=\overline{g}d=
hd_g=d_{hg},$$
so $d_{c\overline{g}f-hg}=0$
and $c\overline{g}f-hg\in K[x^p,y^p]$.
Thus $$c\overline{g}f-hg=
(c\overline{g}f-hg)_{(p)}=
c\overline{g}f_{(p)}-hg_{(p)},$$
so $c\overline{g}(f-f_{(p)})=h(g-g_{(p)})$.
By the minimality of $g-g_{(p)}$ we obtain that
$\overline{g}\in K\setminus\{0\}$
and that $f-f_{(p)}$ is minimal in $K[x,y]^d$.

\medskip

Now, observe that if $f-f_{(p)}$ is divisible by
some non-invertible element $a\in K\setminus\{0\}$,
then $\frac{\partial f}{\partial x}$
and $\frac{\partial f}{\partial y}$
are also divisible by $a$, contrary to
the assumption that $d(x)$ and $d(y)$ are coprime.

\medskip

\noindent
$(4)\Rightarrow (2)$ \
Assume that $d^{\ast}=0$ and that $f-f_{(p)}$
is minimal nonzero polynomial in $K[x,y]^d$,
not divisible by any non-invertible element
of $K\setminus\{0\}$.
We have then $K[x,y]^d\neq K[x^p,y^p]$,
so, by Corollary~\ref{c3},
$d=d_h$ for some $h\in K[x,y]^d\setminus K[x^p,y^p]$.
From Lemma~\ref{l3} we obtain that $\overline{f}\in K[x^p,y^p]$,
where $\overline{f}=\gcd\big(\frac{\partial f}{\partial x},
\frac{\partial f}{\partial y}\big)$.
By Corollary~\ref{c1} we have
$\overline{f}d_h=c'd_f$, where $c'\sim 1$.
Hence $d_{\overline{f}h}=d_{c'f}$,
so $d_{c'f-\overline{f}h}=0$
and $c'f-\overline{f}h\in K[x^p,y^p]$.
Thus $$c'f-\overline{f}h=
(c'f-\overline{f}h)_{(p)}=
c'f_{(p)}-\overline{f}h_{(p)},$$
so $c'(f-f_{(p)})=\overline{f}(h-h_{(p)})$.
By the minimality of $f-f_{(p)}$ we obtain that
$\overline{f}\in K\setminus\{0\}$.
Then, by the assumption, $\overline{f}$ is invertible.
Finally, $d=d_h=cd_f$ for $c=\overline{f}^{-1}c'$.
\end{proof}

\begin{tw}
\label{t2}
Let $K$ be a UFD of characteristic $p>0$.
Let $d$ be a nonzero $K$-derivation
of the polynomial $K$-algebra $K[x,y]$
such that $d(x)$ and $d(y)$ are coprime.
The following conditions are equivalent:

\smallskip

\noindent
$(1)$ \
$K[x,y]^d=K[x^p,y^p,f]$ for some polynomial
$f\in K[x,y]\setminus K[x^p,y^p]$,

\medskip

\noindent
$(2)$ \
$d$ is a Jacobian derivation,

\medskip

\noindent
$(3)$ \
$d^{\ast}=0$ and $K[x,y]^d\neq K[x^p,y^p]$,

\medskip

\noindent
$(4)$ \
$d(f)=0$ and $d(\overline{f})=0$
for some polynomial $f\in K[x,y]\setminus K[x^p,y^p]$,
where $\overline{f}=\gcd\big(\frac{\partial f}{\partial x},
\frac{\partial f}{\partial y}\big)$.
\end{tw}

\begin{proof}
The equivalence $(1)\Leftrightarrow (2)$
and the implication $(1)\Rightarrow (4)$
follow from Proposition~\ref{p1}.
The implication $(3)\Rightarrow (2)$
follows from Corollary~\ref{c3}.
The implication $(2)\Rightarrow (3)$ is obvious.

\medskip

\noindent
$(4)\Rightarrow (3)$ \
Assume that $d(f)=0$ and $d(\overline{f})=0$
for some $f\in K[x,y]\setminus K[x^p,y^p]$.
By Corollary~\ref{c1} we have
$\overline{f}d=cd_f$, where $c\sim 1$, so
$$\overline{f}d^{\ast}=\overline{f}d^{\ast}+d(\overline{f})=
(\overline{f}d)^{\ast}=cd_f^{\ast}=0,$$
that is, $d^{\ast}=0$.
\end{proof}

\medskip

\noindent
{\bf Final remarks.}
Note that the equivalence $(3)\Leftrightarrow (4)$
in Theorem~\ref{t2} can be easily generalized
for arbitrary $n$ in the following way.
If $d$ is an irreducible $K$-derivation of $K[x_1,\dots,x_n]$
and $f_1,\dots,f_{n-1}\in K[x_1,\dots,x_n]^d$
are $p$-independent over $K[x_1^p,\dots,x_n^p]$,
then $d^{\ast}=0\Leftrightarrow d(\overline{F})=0$.
Hence, there is a natural question about a counter-example
to the implication $(3)\Rightarrow (2)$
in such a generalization.

\medskip

Jacobian derivations in characteristic zero
have many nontrivial properties
(see \cite{Daigle}, \cite{MakarLimanov},
sections 3.2 and 3.4 in \cite{Freudenburg}).
Remark that the proof of Corollary~\ref{c2},
based on the properties of divergence,
after an easy adaptation works also
in characteristic $0$.

\end{document}